\documentclass{amsart}
\usepackage{amsmath, amssymb, amsthm}
\usepackage{mathtools}
\usepackage{esint}

\numberwithin{equation}{section}

\theoremstyle{definition}
\newtheorem{definition}{Definition}

\newtheorem{theorem}{Theorem}
\newtheorem{lemma}{Lemma}
\newtheorem{proposition}{Proposition}

\theoremstyle{remark}
\newtheorem{remark}{Remark}

\newtheorem*{acknowledgement}{Acknowledgement}

\title[Limiting behavior of determinantal point processes]{Limiting behavior of determinantal point processes associated with weighted Bergman kernels}
\author{Kiyoon Eum}
\address{Department of Mathematical Sciences, KAIST, 291 Daehak-ro, Yuseong-gu, Daejeon 34141, South Korea}
\email{kyeum@kaist.ac.kr}

\begin{document}

\begin{abstract}
Let $\Omega$ be a bounded pseudoconvex domain in $\mathbb{C}^n$, and let $\phi$ be a strictly plurisubharmonic function on $\Omega$. For each $k\in\mathbb{N}$, we consider determinantal point process $\Lambda_k$ with kernel $K_{k\phi}$, where $K_{k\phi}$ is the reproducing kernel of infinite dimensional weighted Bergman space $H(k\phi)$ with weight $e^{-k\phi}$. We show that the scaled cumulant generating function for $\Lambda_k$ converges as $k\rightarrow\infty$ to a certain limit, which can be explicitly expressed in terms of $\phi$ and a test function $u$. Note that we need to restrict the type of test function $u$ to those that are $\phi$-admissible.
\end{abstract}

\maketitle

\section{Introduction}
In \cite{berman2010growth}, Berman-Boucksom proved that the normalized volume ratio of $L^2$ unit balls in $H^0(X,L^k)$ converges to a certain limit, as $k\rightarrow\infty$. As a corollary, they obtained a limit of scaled cumulant generating functions of determinantal point processes (DPPs) associated with weighted Bergman kernels. 

More precisely, let $(L,h=e^{-\phi})$ be a Hermitian line bundle with positive curvature over a compact complex manifold $X$ with some volume form. The $k^{th}$ power $L^k$ of $L$ inherits metric $e^{-k\phi}$. With the $L^2$-inner product on $H^0(X,L^k)$ induced by $e^{-k\phi}$ and the given volume form, let $K_{k}$ be Bergman type orthogonal projection kernel from the space of all $L^2$-sections onto $H^0(X,L^k)$. For each $K_{k}$ we can associate DPP $\Lambda_k$ on $X$ with kernel $K_{k}$. Then the following result holds (\cite{berman2010growth}, Corollary A, (ii)):
\begin{equation}\label{bb}
\lim_{k\rightarrow\infty}\frac{1}{kN_k}\log \mathbb{E}\left[e^{-k\langle u,\Lambda_k\rangle} \right] =  \mathcal{E}_{eq}(\phi)-\mathcal{E}_{eq}(\phi+u)  
\end{equation}
for any continuous function $u$ on $X$ and $N_k=\dim H^0(X,L^k)$. See \cite{berman2014determinantal} for probabilistic interpretation of Corollary A of \cite{berman2010growth} as presented here. 

The left-hand side of of (\ref{bb}) is a scaled cumulant generating function for $\Lambda_k$, which plays a key role in \cite{berman2014determinantal} in the proof of the large deviation principle for such processes. Since the base manifold $X$ is compact, the spaces of holomorphic sections $H^0(X,L^k)$ are finite dimensional and the proof in particular uses Jacobi's formula for the derivative of the determinant of a finite matrix. 

In this note, we show that the finite-dimensional Jacobi's formula can be replaced with an infinite-dimensional version (\ref{Jac}) involving the Fredholm determinant, yielding a similar limit result (Theorem \ref{main}) in the setting where the Bergman-type kernel is of infinite rank. We consider DPPs $\Lambda_k$ on bounded pseudoconvex domain $\Omega \subset \mathbb{C}^n$, associated with weighted Bergman spaces $H(k\phi)$ with respect to the Lebesgue measure, for a strictly plurisubharmonic weight $\phi$. In this setting, we must restrict the class of test functions; see Definition \ref{u}. 

For recent results on DPP associated with Hilbert spaces of holomorphic functions such as Bergman spaces or Fock spaces, see \cite{bufetov2017determinantal} and \cite{bufetov2022patterson}.

\section{Determinantal point process}
Let $\Omega$ be a locally compact Polish space and $\mu$ be a Radon measure on $\Omega$. A \emph{point process} $\Lambda$ on $\Omega$ is a random integer-valued positive Radon measure on $\Omega$.  If $\Lambda$ almost surely assigns at most measure 1 to singletons, it is called a \emph{simple point process}. 

The \emph{Correlation functions} $\rho_n$ of simple point process $\Lambda$ with respect to the measure $\mu$ is defined (if exist) by
\begin{equation}
\mathbb{E}\left[\prod_{i=1}^{n} \Lambda(D_j) \right] = \int_{\prod_{i=1}^{n}D_j} \rho_n(x_1,...,x_n) d\mu(x_1)...d\mu(x_n)
\end{equation}
for any mutually disjoint subsets $D_1,...,D_n$ of $\Omega$. The correlation functions provide a convenient integral formula for multiplicative functionals of the point process: For any test function $g\in C_c(\Omega)$, 
\begin{equation}
\mathbb{E}\left[\sum_{(\lambda_1,...,\lambda_n)\subset\Lambda}\prod_{k=1}^{n}g(\lambda_k) \right]=\int_{\Omega^n} \rho_{n}(x_1,...,x_n)\prod_{k=1}^{n}g(x_k)d\mu(x_1)...d\mu(x_n).
\end{equation}
As a consequence, we have a following formula.
\begin{proposition}
For any test function $g\in C_{c}(\Omega)$,
\begin{align}\label{correlationfunction}
&\mathbb{E}\left[ \prod_{\lambda\in\Lambda}(1+g(\lambda)) \right] = \sum_{n=0}^{\infty} \frac{1}{n!}\mathbb{E}\left[ \sum_{(\lambda_1,...,\lambda_n)\subset\Lambda}\prod_{k=1}^{n}g(\lambda_k) \right]\\ 
&=\sum_{n=0}^{\infty} \frac{1}{n!}\int_{\Omega^n} \rho_{n}(x_1,...,x_n)\prod_{k=1}^{n}g(x_k)d\mu(x_1)...d\mu(x_n).
\end{align}

\end{proposition}

A point process $\Lambda$ on $\Omega$ is said to be a \emph{determinantal point process (DPP)} with kernel $K$ if it is simple and its correlation functions with respect to the measure $\mu$ satisfy
\begin{equation}\label{corr}
    \rho_n(x_1,...,x_n)=\det\left[K(x_i,x_j)\right]_{1\leq i,j\leq n}
\end{equation}
for every $k\in\mathbb{N}$ and $x_1,...,x_n\in\Omega$. A Theorem of Macchi-Soshnikov \cite{macchi1975coincidence,soshnikov2000determinantal} shows that if $K$ defines a self-adjoint locally trace class integral operator $\mathcal{K}$ on $L^2(d\mu)$, DPP with kernel $K$ exists uniquely if and only if $spec(\mathcal{K})\subset [0,1]$. In that case $\mathcal{K}$ is a finite rank operator with rank $\mathcal{K}\leq N$ if and only if $\mathbb{P}[\Lambda(\Omega)\leq N] = 1$.

\section{Weighted Bergman Space}
Let $d\lambda$ denote the Lebesgue measure on a domain $\Omega$ in $\mathbb{C}^n$. We denote the space of holomorphic functions on $\Omega$ by $\mathcal{O}(\Omega)$. It is well known that the space $H(k\phi):=L^2(e^{-k\phi}d\lambda)\cap\mathcal{O}(\Omega)$, i.e., the space of holomorphic functions that are square-integrable with respect to the weight $e^{-k\phi}$ where $\phi\in C(\Omega)$, is a separable reproducing kernel Hilbert space with reproducing kernel $K_{k\phi}(z,w)$. That is,
\begin{equation}
    f(z)=\int_{\Omega}K_{k\phi}(z,w)f(w)e^{-k\phi(w)}d\lambda(w)=\langle f, K_{k\phi,z} \rangle_{k\phi} 
\end{equation}
for $f\in H(k\phi)$, where $\langle\cdot,\cdot\rangle_{k\phi}$ denotes the $L^2$-inner procuct with respect to measure $e^{-k\phi}d\lambda$. The kernel $K_{k\phi}(z,w)$ is Hermitian positive and can be expanded in terms of orthonormal basis $(e_j)_{j=1}^{\infty}$ of $H(k\phi)$ of the form
\begin{equation}
K_{k\phi}(z,w)=\sum_{j=1}^{\infty} e_j(z)\overline{e_j(w)}  
\end{equation}
where the convergence is locally uniform on $\Omega\times\Omega$ and $L^2(e^{-k\phi}d\lambda)$ in each variable separately. In particular for any compact subset $K$ of $\Omega$ there exists a constant $C_K$ such that
\begin{equation}\label{Fubini}
\sup_{z\in K} \sum_j |e_j(z)|^2 \leq C_K.   
\end{equation}

Let $SPSH(\Omega)=\{\varphi\in C^{\infty}(\Omega) : \partial_{z^j}\overline{\partial_{z^k}}\varphi(z)>0 \;\text{for all}\,z\in\Omega\}$. It is a positve convex cone of strictly plurisubharmonic (psh) funcions on $\Omega$. From now on we always assume $\Omega$ to be a bounded pseudoconvex domain in $\mathbb{C}^n$ and $\phi\in SPSH(\Omega)$. Then we have a following asymptotics of $K_{k\phi}$.
\begin{theorem}[\protect{\cite[Theorem 1]{englivs2002weighted}}]
Let $\Omega$ be a bounded pseudoconvex domain in $\mathbb{C}^n$ and $\phi\in SPSH(\Omega)$. Then
\begin{equation}
\lim_{k\rightarrow\infty} \frac{1}{k^n}K_{k\phi}(z,z)e^{-k\phi(z)}=\frac{1}{\pi^n}\det\left[\partial\overline{\partial}\phi(z)\right]
\end{equation}
for all $z\in\Omega$.
\end{theorem}

In fact we have a following uniform upper bound on $K_{k\phi}$.

\begin{theorem}[\protect{\cite[Lemma 3.1]{berman2009poly}}]\label{upper}
Given any compact subset $K$ of $\Omega$, there exists a constant $C_K$ such that
\begin{equation}
\frac{1}{k^n}K_{k\phi}(z,z)e^{-k\phi(z)} \leq C_K 
\end{equation} on $K$.
The constant $C_K$ depends on $\phi$ and derivatives of $\phi$ continuously.
\end{theorem}
It is proved in \cite{berman2009poly} with a different setting from ours, but the proof shows that it holds in our setting also.

This, with the dominated convergence theorem implies the convergence of the so called Bergman measure (see \cite{berman2009bergman}). We normalize $d^c$ so that $\frac{1}{n!}(dd^c\cdot)^n=\frac{1}{\pi^n}\det(\partial\overline{\partial}\,\cdot)d\lambda$.
\begin{proposition}\label{convBerg}
For any test function $g\in C_c(\Omega)$, 
\begin{equation}
\int_{\Omega}  g(z)\frac{1}{k^n}K_{k\phi}(z,z)e^{-k\phi(z)}d\lambda(z) \rightarrow \frac{1}{n!}\int_{\Omega}g(dd^c \phi)^n    
\end{equation}
as $k\rightarrow\infty$.
\end{proposition} 

We need some operator theoretic notions on Bergman space. For general references, see \cite{simon2005trace} and \cite{zhu2007operator}. First note that integral operator on $L^2(e^{-k\phi}d\lambda)$ with integral kernel $K_{k\phi}(z,w)$ is an orthogonal projection of $L^2(e^{-k\phi}d\lambda)$ onto $H(k\phi)$. We denote this projection operator by $P_{k\phi}$, i.e.,
\begin{equation}
P_{k\phi}f(z)=\int_{\Omega}K_{k\phi}(z,w)f(w)e^{-k\phi(w)}d\lambda(w) 
\end{equation}
for $f\in L^2(e^{-k\phi}d\lambda)$. Given a function $g\in L^{\infty}(\Omega)$, we define an operator $T_{k\phi,g}$ on $H(k\phi)$ by
\begin{equation}
T_{g,k\phi}f=P_{k\phi}(gf),\;\; f\in H(k\phi).
\end{equation}
$T_{k\phi,g}$ is called the Toeplitz operator with symbol $g$. In this note we only concern following two types of (real-valued) symbols : 
\begin{equation}
g=e^{u} \;\text{for}\;\; u\in C_c(\Omega)\;\;\text{or}\;\; g\in C_c(\Omega).
\end{equation}
We will need the following properties of Toeplitz operators with such symbols.
\begin{lemma}\label{lue}
If $g=e^{u}, u\in C_c(\Omega)$, $T_{k\phi,g}$ is invertible. If $g\in C_c(\Omega)$, $T_{k\phi,g}$ is a trace class operator.
\end{lemma}
\begin{proof}
The first claim follows from the fact that $T_{k\phi,e^u}$ is positive definite operator and satisfies $\langle T_{k\phi,e^u}f,f\rangle_{k\phi} \geq C ||f||_{L^2(k\phi)}^2$ with $C=\inf_{\Omega} e^{u}>0$ for all $f\in H(k\phi)$. Second claim is straightforward consequence of (\ref{Fubini}). See also \cite{zhu2007operator}, Proposition 7.11.
\end{proof}

Furthermore we can directly compute the trace and the Fredholm determinant of $T_{k\phi,g}$ for $g\in C_c(\Omega)$, using Berezin transform argument.
\begin{proposition}\label{trdet}
Let $g\in C_c(\Omega)$. Then we have  
\begin{equation}\label{trace}
\text{tr}_{H(k\phi)}(T_{k\phi,g})=\int_{\Omega} K_{k\phi}(z,z)g(z)e^{-k\phi(z)}d\lambda(z) \, ;
\end{equation}
\begin{equation}\label{det}
\det_{H(k\phi)}\left[I+T_{k\phi,g}\right]=\sum_{n=0}^{\infty} \frac{1}{n!} \int_{\Omega^n} \det\left[K_{k\phi}(z_i,z_j)\right]_{i,j} \prod_{l=1}^{n}g(z_l)e^{-k\phi(z_l)}d\lambda(z_1)\cdots d\lambda(z_n).
\end{equation}
\end{proposition}
\begin{proof}
It is well known that for each $n\geq0$, $\wedge^n T_{k\phi,g}$ can be identified with the integral operator on $\wedge^n H(k\phi)$ with the kernel $\frac{1}{n!}\det\left[K_{k\phi}(z_i,w_j)\right]_{i,j}\prod_{l=1}^{n}g(w_l)$.
We need to show that 
\begin{equation}
\text{tr}\wedge^n T_{k\phi,g}=\frac{1}{n!} \int_{\Omega^n} \det\left[K_{k\phi}(z_i,z_j)\right]_{i,j} \prod_{l=1}^{n}g(z_l)e^{-k\phi(z_l)}d\lambda(z_1)\cdots d\lambda(z_n)
\end{equation}
for all $n\geq0$. It follows from trace formula in Bergman space if we consider $\wedge^n H(k\phi)$ as a subspace of weighted Bergman space on $\Omega^n$ with kernel $\prod_{l=1}^{n}K_{k\phi}(z_l,w_l)$. See Theorem 6.4 and its corollary in \cite{zhu2007operator}.
\end{proof}

Note that they are the same as classical formulae for integral operators on $L^2$-space. We will use this fact to replace the $L^2$-Fredholm determinant in the classical formulation of DPP theory (see, for example \cite{soshnikov2000determinantal} and \cite{shirai2003random}) to the $H(k\phi)$-Fredholm determinant.

\begin{remark}
Given $\phi\in SPSH(\Omega)$ and $u\in C_c(\Omega)$, consider $H(k\phi)$ and $H(k(\phi+u))$. It is clear that $\langle\cdot,\cdot\rangle_{k\phi}$ and $\langle\cdot,\cdot\rangle_{k(\phi+u)}$ induce equivalent inner product on space of holomorphic functions (if they are finite). Thus $H(k\phi)=H(k(\phi+u))$ as a vector space and two inner products are equivalent. By Lidskii's theorem we know that trace of trace class operator on seperable Hilbert space is a sum of its eigenvalues, hence it is independent of the inner product. Thus in Proposition \ref{trdet} we can replace tr$_{H(k\phi)}$ and $\det_{H(k\phi)}$ by tr$_{H(k(\phi+u))}$ and $\det_{H(k(\phi+u))}$.
\end{remark}

\section{Main Result}

We start with the definition of our DPP.
\begin{definition}\label{DPP}
Let $\Omega$ be a bounded pseudoconvex domain in $\mathbb{C}^n$ and $\phi\in SPSH(\Omega)$. Let $K_{k\phi}(z,w)$ be a reproducing kernel of $H(k\phi)$. We define $\Lambda_k$ as a DPP with kernel $K_{k\phi}$ with respect to $d\mu=e^{-k\phi}d\lambda$. Its existence and uniqueness is guaranteed by the Macchi-Soshnikov theorem and Mercer's theorem.
\end{definition}

Substituting $g=e^{-ku}-1$ and (\ref{corr}) with $K=K_{k\phi}$ into (\ref{correlationfunction}) and comparing it with (\ref{det}), we get
\begin{theorem}\label{Fred}
Let $\Lambda_k$ as in Definition \ref{DPP}. Then for any $u\in C_c(\Omega)$,
\begin{equation}\label{fred}
\mathbb{E}\left[e^{-\langle u,\Lambda_k\rangle} \right] =\det_{H(k\phi)}\left[I+T_{k\phi,(e^{-u}-1)}\right]
\end{equation}
where $\langle u,\Lambda_k\rangle$ denotes the canonical pairing of test functions and Radon measures, i.e. $\langle u,\Lambda_k\rangle=\int_{\Omega} ud\Lambda_k$.
\end{theorem}

In the proof of the main result, we need to take the derivative of log of (\ref{fred}) where the test function $u_t$ depends smoothly on $t$. Thus we need a following infinite dimensional Jacobi's formula for trace class operators.
\begin{proposition}[\protect{\cite[Corollary 5.2]{simon2005trace}}]
Let $f : A \mapsto \det[I+A]$ for trace class operators $A$. Then it is Fr\'echet differentiable and its derivative is given by
\begin{equation}
Df(A)=\det[I+A]\text{tr}[(I+A)^{-1}\,\cdot\,]
\end{equation}
for $A$ such that $-1\notin \sigma(A)$. In particular 
\begin{equation}\label{Jac}
\frac{d}{dt} \log\det[I+A_t] = \text{tr}\left[(I+A_t)^{-1}\frac{dA_t}{dt}\right]    
\end{equation}
if $I+A_t$ is invertible.
\end{proposition}

We will identify the operator $(I+A_t)^{-1}\frac{dA_t}{dt}$ in (\ref{Jac}) in the proof, by the following key lemma.

\begin{lemma}[\protect{\cite[Proposition 1]{englivs2008toeplitz}}]\label{keylem}
Let $u\in C_c(\Omega)$. Then
\begin{equation}
T^{-1}_{k\phi,e^{-ku}}K_{k\phi,w}(z)=K_{k(\phi+u),w}(z)    
\end{equation}
where $T^{-1}_{k\phi,e^{-ku}}$ acts on $K_{k\phi,w}$ considered as a function of $z$.  
\end{lemma}
\begin{proof}
By Lemma \ref{lue}, $T^{-1}_{k\phi,e^{-ku}}$ is invertible. Let $h=T^{-1}_{k\phi.e^{-ku}}K_{k\phi,w}\in H(k\phi)$. Then for any $f\in H(k\phi)=H(k(\phi+u))$, we have
\begin{align}
\langle f,K_{k(\phi+u),w} \rangle_{k(\phi+u)}=f(w)=\langle f,K_{k\phi,w} \rangle_{k\phi}=\langle f,T_{k\phi,e^{-ku}}h \rangle_{k\phi} = \langle f,P_{k\phi}e^{-ku}h \rangle_{k\phi}\\
= \langle P_{k\phi}f,e^{-ku}h \rangle_{k\phi}
= \langle f,e^{-ku}h \rangle_{k\phi}
=\int_{\Omega}f(z)\overline{h(z)}e^{-k(\phi(z)+u(z))}d\lambda(z)=\langle f,h \rangle_{k(\phi+u)}.
\end{align}
Note that in the computation we consider $f$ and $e^{-ku}h$ as elements in $L^2(e^{-k\phi}d\lambda)$. This proves the claim. In fact in \cite{englivs2008toeplitz} this is proved for more general types of operators.
\end{proof}

Finally, to state the main result, we need to restrict the class of test functions. The reason for this restriction is explained in the remark following Proposition \ref{eq}.

\begin{definition}\label{u}
Given $\phi\in SPSH(\Omega)$, $u\in C^{\infty}_c(\Omega)$ is called $\phi$-admissible test function if $\phi+u\in SPSH(\Omega)$.    
\end{definition}

For example, if $\text{supp}(u)\subset\subset\Omega$ and $\sup_{\Omega}|D^2u|$ is less than minimum eigenvalue of $\partial\overline{\partial}\phi$ in $\text{supp}(u)$ then $u$ is $\phi$-admissible. It is clear that if $u$ is $\phi$-admissible, $tu$ is $\phi$-admissible for all $t\in[0,1]$. 

For functions of the form $\phi+u$, with $\phi$-admissible $u$, we introduce the Monge-Amp\`ere energy functional.
\begin{definition}
Let $u$ be $\phi$-admissible. The energy $\mathcal{E}(\phi+u)$ is defined by
\begin{equation}\label{en}
\mathcal{E}(\phi+u)=\frac{1}{(n+1)!}\sum_{j=0}^{n}\int_{\Omega}u(dd^c(\phi+u))^j\wedge(dd^c\phi)^{n-j}.
\end{equation}
It is finite since $u\in C_c(\Omega)$. It is normalized so that $\mathcal{E}(\phi)=0$.
\end{definition}
The functional $\mathcal{E}$ is a primitive of the normalized Monge-Amp\`ere operator in the following sense.
\begin{proposition}\label{eq}
Let $u$ be $\phi$-admissible and $t\in[0,1]$. Then
\begin{equation}
\left.\frac{d}{dt}\right|_{t} \mathcal{E}(\phi+tu) = \frac{1}{n!}\int_{\Omega}u(dd^c (\phi+tu))^n 
\end{equation}
\end{proposition}
\begin{proof}
We prove the following identity:
\begin{equation}\label{bc}
\left.\frac{d}{dt}\right|_{t} \frac{1}{l!}\sum_{j=0}^{l-1} tu (dd^c(\phi+tu))^j\wedge (dd^c\phi)^{l-1-j}=\frac{1}{(l-1)!}u(dd^c(\phi+tu))^{l-1}.
\end{equation}
The result follows from (\ref{bc}) by substituting $l=n+1$ and integrating over $\Omega$ (which is finite since $u\in C_c(\Omega)$).
Note that the following proof essentially computes the Bott-Chern forms for Chern character polynomials for line bundles (see \cite[(2.20)]{eum2025partition}).
\begin{align*}
&\left.\frac{d}{dt}\right|_{t} \frac{1}{l!}\sum_{j=0}^{l-1} tu (dd^c(\phi+tu))^j\wedge (dd^c\phi)^{l-1-j}=\\
&=\frac{1}{l!} \sum_{j=0}^{l-1}u (dd^c(\phi+tu))^j\wedge (dd^c\phi)^{l-1-j}\\
&\quad+\frac{1}{l!}\sum_{j=1}^{l-1}uj(tdd^cu)\wedge(dd^c(\phi+tu))^{j-1}\wedge (dd^c\phi)^{l-1-j} \\
&=\frac{1}{l!} \sum_{j=0}^{l-1}u (dd^c(\phi+tu))^j\wedge (dd^c\phi)^{l-1-j} \\
&\quad +\frac{1}{l!}\sum_{j=1}^{l-1}uj(dd^c(\phi+tu)-
dd^c\phi)\wedge(dd^c(\phi+tu))^{j-1}\wedge (dd^c\phi)^{l-1-j} \\
&=\frac{1}{l!} \sum_{j=0}^{l-1}u (dd^c(\phi+tu))^j\wedge (dd^c\phi)^{l-1-j} \\
&\quad +\frac{1}{l!}\sum_{j=1}^{l-1}uj(dd^c(\phi+tu))^{j}\wedge (dd^c\phi)^{l-1-j}-\frac{1}{l!}\sum_{j=1}^{l-1}uj(dd^c(\phi+tu))^{j-1}\wedge (dd^c\phi)^{l-j} \\
&=\frac{1}{l!}\sum_{j=0}^{l-1}u (dd^c(\phi+tu))^j\wedge (dd^c\phi)^{l-1-j}-\frac{1}{l!}\sum_{j=1}^{l-2}u(dd^c(\phi+tu))^{j}\wedge (dd^c\phi)^{l-1-j}\\
&\quad +\frac{(l-1)}{l!}u(dd^c(\phi+tu))^{l-1}-\frac{1}{l!}u(dd^c\phi)^{l-1}\\
&=\frac{1}{(l-1)!}u(dd^c(\phi+tu))^{l-1}.
\end{align*}
\end{proof}

\begin{remark}
In view of results in \cite{berman2010growth,berman2009bergman,berman2014determinantal}, it is natural to expect that we can consider arbitrary $C^{\infty}_c(\Omega)$ test function $u$ and use plurisubharmonic envelope (psh projection) to define $\mathcal{E}_{eq}$ and prove similar results. More precisely, we could use $P\varphi=\left(\sup\{\psi\in PSH(\Omega): \psi\leq\varphi \;\;\text{on}\;\; \Omega\}\right)^*$ to define $\mathcal{E}_{eq}(\phi+u):=\mathcal{E}(P(\phi+u))$. However in general $P(\phi+u)-\phi$ is not compactly supported nor continuous \cite{walsh1968continuity}, hence arguments using the fact that $u\in C_c(\Omega)$ break down in this case. A more natural domain of energy functionals such as $\mathcal{E}$ on $\Omega$ would be Cegrell class \cite{aahag2012dirichlet}, but then we have to impose further technical assumptions on $\phi$.
\end{remark}

Now we prove our main result.
\begin{theorem}\label{main}
Let $\Lambda_k$ be DPP defined in Definition \ref{DPP}, $\phi\in SPSH(\Omega)$ and $u$ be $\phi$-admissible. Then we have
\begin{equation}
\lim_{k\rightarrow\infty}\frac{1}{k^{n+1}}\log \mathbb{E}\left[e^{-k\langle u,\Lambda_k\rangle} \right] = \frac{-1}{(n+1)!}\sum_{j=0}^{n}\int_{\Omega}u(dd^c(\phi+u))^j\wedge(dd^c\phi)^{n-j}.
\end{equation}
\end{theorem}
\begin{proof}
By Theorem \ref{Fred}, we have    
\begin{equation}
\frac{1}{k^{n+1}}\log \mathbb{E}\left[e^{-k\langle u,\Lambda_k\rangle} \right] =\frac{1}{k^{n+1}}\log\det_{H(k\phi)}\left[I+P_{k\phi}(e^{-ku}-1)\right].
\end{equation}
Let $t\in[0,1]$. We claim that
\begin{align}
&\left.\frac{d}{dt}\right|_{t}\frac{1}{k^{n+1}}\log\det_{H(k\phi)}\left[I+P_{k\phi}(e^{-ktu}-1)\right] \nonumber\\
&=\int_{\Omega}\frac{K_{k(\phi+tu)}(w,w)}{k^n}(-u(w))e^{-k(\phi(w)+tu(w))}d\lambda(w).\label{deriv}
\end{align}
First we prove this claim. By (\ref{Jac}), 
\begin{align}
&\left.\frac{d}{dt}\right|_{t} \log\det_{H(k\phi)}\left[I+P_{k\phi}(e^{-ktu}-1)\right] \\
&=\text{tr}_{H(k\phi)}\left[(I+P_{k\phi}(e^{-ktu}-1))^{-1}P_{k\phi}(-kue^{-ktu}) \right].\label{tr}
\end{align}
Since $P_{k\phi}=I$ on $H(k\phi), \,(I+P_{k\phi}(e^{-ktu}-1))^{-1}=(P_{k\phi}e^{-ktu})^{-1}=T^{-1}_{k\phi,e^{-ktu}}$ on $H(k\phi)$.

Now we will identify the operator inside of tr in (\ref{tr}) as an integral operator on $H(k\phi)$, which is essentially a Toeplitz operator on $H(k(\phi+tu))$ with symbol $-ku$.
Let $(e_j)_j$ be an orthonormal basis for $H(k\phi)$. For $f\in H(k\phi)$, 
\begin{align}
&P_{k\phi}(-kue^{-ktu})f(z)=\int_{\Omega} K_{k\phi}(z,w)(-ku(w)e^{-ktu(w)})f(w)e^{-k\phi(w)}d\lambda(w) \\
&=\int_{\Omega} \sum_{j}e_j(z)\overline{e_j(w)}(-ku(w)e^{-ktu(w)})f(w)e^{-k\phi(w)}d\lambda(w) \\
&=\sum_{j} e_j(z)\int_{\Omega}\overline{e_j(w)}(-ku(w)e^{-ktu(w)})f(w)e^{-k\phi(w)}d\lambda(w).
\end{align}
Hence 
\begin{align}
&T^{-1}_{k\phi,e^{-ktu}}P_{k\phi}(-kue^{-ktu})f(z) \\
&=\sum_{j} \left(T^{-1}_{k\phi,e^{-ktu}}e_j\right)(z)\int_{\Omega}\overline{e_j(w)}(-ku(w)e^{-ktu(w)})f(w)e^{-k\phi(w)}d\lambda(w)\\
&=\int_{\Omega}\sum_{j} \left(T^{-1}_{k\phi,e^{-ktu}}e_j\right)(z)\overline{e_j(w)}(-ku(w)e^{-ktu(w)})f(w)e^{-k\phi(w)}d\lambda(w).
\end{align}
Interchange of the order of integration, summation and action of bounded operator can be justified since by (\ref{Fubini}),

\begin{align}
&\sum_{j}\left|\int_{\Omega}\overline{e_j(w)}(-ku(w)e^{-ktu(w)})f(w)e^{-k\phi(w)}d\lambda(w)\right|^2\\
&\leq \sum_{j}C\int_{\text{supp}(u)}|e_j(w)|^2d\lambda(w)=C\int_{\text{supp}(u)}\sum_{j}|e_j(w)|^2d\lambda(w)<\infty
\end{align}
for some constant $C$.

By Lemma \ref{keylem}, 
\begin{equation}
\sum_{j} \left(T^{-1}_{k\phi,e^{-ktu}}e_j\right)(z)\overline{e_j(w)}=T^{-1}_{k\phi,e^{-ktu}}K_{k\phi}(z,w)=K_{k(\phi+tu)}(z,w).
\end{equation}
This implies
\begin{align}
&T^{-1}_{k\phi,e^{-ktu}}P_{k\phi}(-kue^{-ktu})f(z) \\
&=\int_{\Omega}K_{k(\phi+tu)}(z,w)(-ku(w)e^{-ktu(w)})f(w)e^{-k\phi(w)}d\lambda(w)\\
&=\int_{\Omega}K_{k(\phi+tu)}(z,w)(-ku(w))f(w)e^{-k(\phi(w)+tu(w))}d\lambda(w)\\
&=T_{k(\phi+tu),-ku}f(z).
\end{align}
Thus by trace formula (\ref{trace}),
\begin{align}
&\frac{1}{k^{n+1}}\text{tr}_{H(k\phi)}\left[T^{-1}_{k\phi,e^{-ktu}}P_{k\phi}(-kue^{-ktu}) \right]
=\frac{1}{k^{n+1}}\text{tr}_{H(k(\phi+tu))}T_{k(\phi+tu),-ku}\\
&=\frac{1}{k^{n+1}}\int_{\Omega}K_{k(\phi+tu)}(z,z)(-ku(z))e^{-k(\phi(z)+tu(z))}d\lambda(z)\\
&=\int_{\Omega}\frac{K_{k(\phi+tu)}(z,z)}{k^n}(-u(z))e^{-k(\phi(z)+tu(z))}d\lambda(z).
\end{align}
This proves (\ref{deriv}). By integrating (\ref{deriv}) over $0\leq t\leq 1$, we get 
\begin{align}
&\frac{1}{k^{n+1}}\log \mathbb{E}\left[e^{-k\langle u,\Lambda_k\rangle} \right]=\int_{0}^{1}dt\left.\frac{d}{dt}\right|_{t} \frac{1}{k^{n+1}}\log \mathbb{E}\left[e^{-k\langle tu,\Lambda_k\rangle} \right]\\
&=\int_{0}^{1}dt\int_{\Omega}\frac{K_{k(\phi+tu)}(z,z)}{k^n}(-u(z))e^{-k(\phi(z)+tu(z))}d\lambda(z).
\end{align}
Proposition \ref{convBerg} implies 
\begin{equation}\label{convber}
\int_{\Omega}\frac{K_{k(\phi+tu)}(z,z)}{k^n}(-u(z))e^{-k(\phi(z)+tu(z))}d\lambda(z) \rightarrow \frac{1}{n!}\int_{\Omega}(-u)(dd^c (\phi+tu))^n   
\end{equation}
as $k\rightarrow\infty$. Since $u\in C_c(\Omega)$ the integrand in the LHS of (\ref{convber}) is uniformly bounded, by Theorem \ref{upper}. By dominated convergence theorem and Proposition \ref{eq}
\begin{equation}
\frac{1}{k^{n+1}}\log \mathbb{E}\left[e^{-k\langle u,\Lambda_k\rangle} \right] \rightarrow \int_{0}^{1}dt\frac{1}{n!}\int_{\Omega}(-u)(dd^c (\phi+tu))^n = -\mathcal{E}(\phi+u).
\end{equation}
Formula (\ref{en}) of $\mathcal{E}$ then completes the proof.
\end{proof}

\bibliographystyle{alpha}
\bibliography{References}

\begin{thebibliography}{{\AA}CC12}

\bibitem[{\AA}CC12]{aahag2012dirichlet}
Per {\AA}hag, Urban Cegrell, and Rafa{\l} Czy{\.z}.
\newblock On {D}irichlet's principle and problem.
\newblock {\em Mathematica scandinavica}, pages 235--250, 2012.

\bibitem[BB10]{berman2010growth}
Robert Berman and S{\'e}bastien Boucksom.
\newblock Growth of balls of holomorphic sections and energy at equilibrium.
\newblock {\em Inventiones mathematicae}, 181(2):337--394, 2010.

\bibitem[Ber09a]{berman2009bergman}
Robert~J Berman.
\newblock Bergman kernels and equilibrium measures for line bundles over projective manifolds.
\newblock {\em American journal of mathematics}, 131(5):1485--1524, 2009.

\bibitem[Ber09b]{berman2009poly}
Robert~J Berman.
\newblock Bergman kernels for weighted polynomials and weighted equilibrium measures of {C}n.
\newblock {\em Indiana University mathematics journal}, pages 1921--1946, 2009.

\bibitem[Ber14]{berman2014determinantal}
Robert~J Berman.
\newblock Determinantal point processes and fermions on complex manifolds: large deviations and bosonization.
\newblock {\em Communications in Mathematical Physics}, 327:1--47, 2014.

\bibitem[BQ17]{bufetov2017determinantal}
Alexander~I Bufetov and Yanqi Qiu.
\newblock Determinantal point processes associated with {H}ilbert spaces of holomorphic functions.
\newblock {\em Communications in Mathematical Physics}, 351(1):1--44, 2017.

\bibitem[BQ22]{bufetov2022patterson}
Alexander~I Bufetov and Yanqi Qiu.
\newblock The {P}atterson--{S}ullivan reconstruction of pluriharmonic functions for determinantal point processes on complex hyperbolic spaces.
\newblock {\em Geometric and Functional Analysis}, 32(2):135--192, 2022.

\bibitem[Eng02]{englivs2002weighted}
Miroslav Engli{\v{s}}.
\newblock Weighted {B}ergman kernels and quantization.
\newblock {\em Communications in mathematical physics}, 227(2):211--241, 2002.

\bibitem[Eng08]{englivs2008toeplitz}
Miroslav Engli{\v{s}}.
\newblock Toeplitz operators and weighted {B}ergman kernels.
\newblock {\em Journal of Functional Analysis}, 255(6):1419--1457, 2008.

\bibitem[Eum25]{eum2025partition}
Kiyoon Eum.
\newblock Partition functions of determinantal point processes on polarized {K}\"ahler manifolds.
\newblock {\em arXiv preprint arXiv:2503.01524}, 2025.

\bibitem[Mac75]{macchi1975coincidence}
Odile Macchi.
\newblock The coincidence approach to stochastic point processes.
\newblock {\em Advances in Applied Probability}, 7(1):83--122, 1975.

\bibitem[Sim05]{simon2005trace}
Barry Simon.
\newblock {\em Trace ideals and their applications}, volume 120.
\newblock American Mathematical Society, 2005.

\bibitem[Sos00]{soshnikov2000determinantal}
Alexander Soshnikov.
\newblock Determinantal random point fields.
\newblock {\em Russian Mathematical Surveys}, 55(5):923, 2000.

\bibitem[ST03]{shirai2003random}
Tomoyuki Shirai and Yoichiro Takahashi.
\newblock Random point fields associated with certain {F}redholm determinants {I}: fermion, {P}oisson and boson point processes.
\newblock {\em Journal of Functional Analysis}, 205(2):414--463, 2003.

\bibitem[Wal68]{walsh1968continuity}
John~B Walsh.
\newblock Continuity of envelopes of plurisubharmonic functions.
\newblock {\em Journal of Mathematics and Mechanics}, 18(2):143--148, 1968.

\bibitem[Zhu07]{zhu2007operator}
Kehe Zhu.
\newblock {\em Operator theory in function spaces}, volume 138.
\newblock American Mathematical Society, 2007.

\end{thebibliography}

\begin{acknowledgement}
This work was partially supported by the National Research Foundation of Korea (NRF) grant funded by the Korea government(MSIT) RS-2024-00346651.
\end{acknowledgement}

\end{document}